\documentclass{article}
\usepackage{amsmath}
\usepackage{amssymb}
\usepackage{eucal}

\newtheorem{theorem}{Theorem}[section]

\newtheorem{corollary}[theorem]{Corollary}

\newtheorem{definition}[theorem]{Definition}
\newtheorem{example}[theorem]{Example}

\newtheorem{lemma}[theorem]{Lemma}

\newenvironment{proof}[1][Proof]{\textbf{#1 }}

\begin{document}
\baselineskip 15pt

\title{On $p$-nilpotency of finite groups\thanks{The
 work were supported by the NNSF of P. R. China (Grant 11071229) and
  the Scientific Research Foundation of CUIT(Grant J201114)}}
\author{Baojun Li$^1$ and Tuval Foguel$^2$\\
{\small 1. College of Mathematics,
Chengdu University of Information Technology }\\
 {\small Chengdu Sichuan 610225, P. R. China}\\
 {\small 2.Department of Mathematics and Computer Science, Western Carolina University}\\
{\small Cullowhee, NC 28723 USA}\\
{\small\  E-mail: baojunli@cuit.edu.cn; tsfoguel@wcu.edu}}

\date{}
\maketitle

\abstract Let $H$ be a subgroup  of a group $G$. $H$ is said
  satisfying $\Pi$-property in $G$, if $|G/K:N_{G/K}(HK/K\cap L/K)|$ is a
 $\pi(HK/K\cap L/K))$-number for any chief factor $L/K$ of $G$, and,
 if there is a subnormal supplement $T$ of $H$ in $G$ such
that $H\cap T\le I\le H$ for some subgroup $I$ satisfying
$\Pi$-property in $G$, then $H$ is said $\Pi$-normal in $G$. By
these properties of some subgroups, we obtain some new criterions of
$p$-nilpotency of finite groups.

\section{Introduction}
Throughout this paper, all groups  considered  are finite. Let $G$
be a group. $\pi(G)$ denotes the set of all prime divisors of $|G|$ and
$\pi$ denotes a subset of $\mathbb{P}$, the set of all primes, and $\pi'$ is
the complement of $\pi$ in $\mathbb{P}$.
 An integer $n$ is called a
$\pi$-number if all its prime divisors belong to $\pi$.

 Recall that a subgroup $A$ of a group $G$ is said to permute with a subgroup $B$
  if $AB =BA$. It is known that $AB$ is a subgroup of $G$ if and only if $A$
  permutes with $B$. Thus the permutability of subgroups is very important and is researched extensively. In the literature,
 seminormal subgroups  \cite{Fo,Su}, S-quasinormally
embedded subgroups \cite{BB-embedded}, $X$-permutable subgroups
 \cite{X-per}, c-normal subgroups  \cite{C-normal}, weakly s-permutable subgroups  \cite{Skiba}, and so on have been studied by many mathematicians.
To discuss on some essential properties of  these generalized permutabilities
of subgroups, the following definitions were arisen in \cite{pi}.

\begin{definition} Let $H$ be a subgroup of $G$. We call that $H$
has $\Pi$-property in $G$ if for any $G$-chief factor $L/K$,
$|G/K:N_{G/K}(HK/K\cap L/K)|$ is a $\pi(HK/K\cap L/K)$-number.
\end{definition}

\begin{definition} Let $H$ be a subgroup of $G$. If there is a subgroup $T$ of
$G$ such that $HT=G$ and  $H\cap T\le I\le H$, where $I$ has
$\Pi$-property in $G$, then $H$ is called $\Pi$-supplemented in $G$.
If, furthermore, $T$ is subnormal in $G$, then $H$ is called
$\Pi$-normal in $G$.
\end{definition}

By the work in \cite{pi}, we known that $\Pi$-property is an essential property of many  generalized permutabilities
of subgroups and many known results were uniformed by using $\Pi$-property and $\Pi$-normality of subgroups. In this paper, we shall do some
further research and give some new criterions of $p$-nilpotency of groups by the $\Pi$-normality of primary
subgroups of given order(primary subgroups of given order $|D|$ was first studied by A. Skiba (cf\cite{Skiba})).
 The main results in this paper is

\noindent{\bf Theorem A.} {\it
    Let $p$ be an odd prime and $N$ a normal subgroup of $G$ with $p$-nilpotent quotient. Assume
    that $P$ is a Sylow $p$-subgroup of $N$ and $N_G(P)$ is $p$-nilpotent. If there is an integer
     $m$ with $1<p^m<|P|$ such that every subgroup of $P$ of order $p^m$ not having $p$-nilpotent
      supplement in $G$ is $\Pi$-normal in $G$, then $G$ is $p$-nilpotent.}

 \noindent{\bf Theorem B.} {\it
    Let be $G$ be a group and $p$ a prime with $(|G|, p-1)=1$. Assume that $N$ is a normal subgroup
    of $G$ with $p$-nilpotent quotient and $P$  a Sylow $p$-subgroup of $N$. Suppose that there is
    an integer  $m$ with $1<p^m<|P|$ such that every subgroup of $P$ of order $p^m$ not having $p$-nilpotent
    supplement in $G$ is $\Pi$-normal in $G$, then $G$ is $p$-nilpotent if one of the following conditions holds:}

\noindent(i) $m\ge 2$;

\noindent(ii)  {\it $P$ is abelian  or $p>2$ is an odd prime;}

\noindent(iii) {\it every cyclic subgroup  of $P$ of order 4 not having $p$-nilpotent supplement in $G$ is $\Pi$-normal
in $G$;}

\noindent(iv) {\it $N$ is soluble and $P$ is quaternion-free.}

\section{Preliminaries}
 In this section, we list some lemmas which should be used in the proofs of our main results, and introduce some
  notions and terminologies. For notations and terminologies not given, the reader is referred to the texts
  of W. Guo \cite{GUO-BOOK} or B. Huppert \cite{Huppert-BOOK-I}

\begin{lemma}{\rm (\cite[Proposition 2.1]{pi})}\label{1}
Let $H$ be a subgroup of $G$ and $N$ a normal subgroup of $G$.

{\rm (1)} If $H$ has $\Pi$-property in $G$, then $HN/N$ has
$\Pi$-property in $G/N$.

{\rm (2)} If $H$ has $\Pi$-property in $G$, then $H$ is both
$\Pi$-normal and $\Pi$-supplemented in $G$.

 {\rm (3)} If $H$ is $\Pi$-normal $(\Pi$-supplemented$)$ in $G$,
 then   $HN/N$ is $\Pi$-normal $(\Pi$-supplemented$)$ in $G/N$ when
 $N\subseteq H$ or $(|H|,|N|)=1$.
\end{lemma}

\begin{lemma}\label{B-p}
  Let $p$ be a prime and $N$ a normal subgroup of $G$ with a $p$-nilpotent quotient.  Assume that $P$ is a
  Sylow $p$-subgroup of $N$ and $N_G(P)=C_G(P)$. Then $G$ is $p$-nilpotent.
\end{lemma}
\begin{proof}
    Since $N_N(P)=N\cap N_G(P)=N\cap C_G(P)=C_N(P)$, $N$ is $p$-nilpotent by Burnside
    Theorem (cf. \cite[Ch 7, Theorem 4.3]{Gorenstein-BOOK}). Let $Q$ be the normal $p$-complement of $N$.
    Then the hypotheses still hold on $G/Q$. If $Q\ne 1$ then $G/Q$ is $p$-nilpotent by induction on $|G|$
    and so $G$ is $p$-nilpotent. Assume $Q=1$. Then $N=P$ and $G=N_G(P)=C_G(P)$, that is $P\le Z(G)$.
    Since $G/N$ is $p$-nilpotent, we see that $G$ is $p$-nilpotent and the lemma holds.
\end{proof}

Recall, a generalized Fitting subgroup of a group $G$, denoted by
$F^*(G)$, is the maximal quasinilpotent normal subgroup of $G$ and
$F^*(G)=F(G)$ if and only if $F^*(G)$ is soluble (cf.\cite[Chapter X
$, \S$13]{Huppert-BOOK-III}). The supersoluble-hypercentral,
$Z_\infty^\mathcal{U}(G)$, of a group $G$ is the maximal normal
subgroup of $G$ with all $G$-chief factor lying in it is cyclic (it
is also called the supersoluble embedded subgroup of $G$ in
\cite{Weinstein}).

\begin{lemma}{\rm(\cite[Lemma 2.17]{pi})}\label{Fit} Let $G$ be a group and $E$ a  normal subgroup
of $G$. If $F^*(E)\subseteq Z_\infty^\mathcal{U}(G)$, then
$E\subseteq Z_\infty^\frak{U}(G)$.
\end{lemma}

\begin{lemma}\label{2}{\rm(\cite[Proposition 2.7]{pi})} Let $H$ be a $p$-subgroup of $G$ and $N$ a minimal normal
 subgroup of $G$. Assume that $H$ is
$\Pi$-normal in $G$. If there is a Sylow p-subgroup $G_p$ of $G$ such that $H \cap N \unlhd G_p$,
then $H \cap N = N$ or 1.
\end{lemma}

\begin{lemma}{\rm (\cite[Proposition 2.8]{pi})}\label{12}
Let $H$ be a $p$-subgroup of $G$ for some prime divisor $p$ of $|G|$
and $L$ a minimal normal subgroup of $G$. Assume that $H$ is
$\Pi$-normal in $G$. Then $L$ is a $p$-group if  $H\cap L\ne 1$.
\end{lemma}

\begin{lemma}\label{p-s}
    Let $p$ be an odd prime and  $L$ a normal subgroup of $G$ with $G/L$ is a $p$-group. Suppose $P$
    is a Sylow $p$-subgroup of $G$ and $N_G(P)$ is $p$-nilpotent. If every maximal subgroup  of $P$
   either has a $p$-nilpotent supplement in $G$ or satisfies $P_1\cap L=1$, then $G$ is $p$-nilpotent.
\end{lemma}
\begin{proof}
 If $O_{p'}(G)\ne 1$ then one can obtained by induction on $|G|$ that $G/O_{p'}(G)$ is $p$-nilpotent and so
 is $G$. Assume that $O_{p'}(G)=1$.
 Let $M$ be a proper subgroup of $G$ containing $P$. Then $M/M\cap L$ is a $p$-group and $N_M(P)=N_G(P)\cap M$
 is $p$-nilpotent. Let $P_1$ be a maximal subgroup of $P$ with $P_1\cap M\cap L\ne 1$. Then $P_1\cap L\ne 1$
 and so $P_1$ has a $p$-nilpotent supplement in $G$. It follows that $P_1$ has a $p$-nilpotent supplement in $M$.
 Thus the hypotheses hold on $M$ and hence we can have that $M$ is $p$-nilpotent by induction on $|G|$.

If for any characteristic subgroup $X$ of $P$, $N_G(X)$ is $p$-nilpotent, then $G$ is $p$-nilpotent by
\cite[Corollary]{Thomposon}. Assume that there is a characteristic subgroup $X$ of $P$ with $N_G(X)$ is
not $p$-nilpotent. Since $N_G(P)$ is $p$-nilpotent, we can choose $X$ to be with maximal order. Since
$P\subseteq N_G(X)$, if $N_G(X)<G$ then $N_G(X)$ is $p$-nilpotent by above argument. So $N_G(X)=G$
and $X\unlhd G$. The maximality of $X$ also implies that for any characteristic subgroup $Y/X$ of $P/X$,
$N_{G/X}(P/X)$ is $p$-nilpotent. Again by \cite[Corollary]{Thomposon}, $G/X$ is $p$-nilpotent. In particularly,
$G$ is $p$-soluble.

Let $R$ be a minimal normal subgroup of $G$. Then $R$ is a $p$-group since $G$ is $p$-soluble and $O_{p'}(G)=1$.
If $P_1/R$ is a maximal subgroup of $P/R$ with $P_1/R\cap LR/R\ne 1$ then $P_1\cap LR=(P_1\cap L)R\nsubseteq R$
and hence $P_1\cap L\ne 1$. Then, one can see that the hypotheses still hold on $G/R$ and by induction on $|G|$,
we have that $G/R$ is $p$-nilpotent. Since the class of all $p$-nilpotent group is a saturated formation,  $R$ is
the only minimal subgroup of $G$ and $R\cap \Phi(G)=1$. Thus there is a $p$-nilpotent complement $T$ of $R$ in $G$.
Let $Q$ be the normal $p$-complement of $T$. Then $Q\unlhd T$. If $Q\unlhd G$ then $Q\subseteq O_{p'}(G)=1$ and so
$G$ is a $p$-group. Assume that $Q$ is not normal in $G$. Then $T=N_G(Q)$. Let $P_1$ be a maximal subgroup of $P$
containing $P\cap T$. Since $R(P\cap T)=P\cap RT=P$, $P_1=P_1\cap R(P\cap T)=(R\cap P_1)(P\cap T)$ and $R\cap P_1<R$.
If $|R|\ne p$, then $R\cap P_1\ne 1$ and hence $P_1\cap L\ne 1$. Thus $P_1$ has a $p$-nilpotent supplement $U$ in $G$.
Since $G$ is $p$-soluble, we can assume that $Q\subseteq U$ and hence $U\subseteq N_G(Q)=T$. It follows that
$G=P_1U=P_1T=(R\cap P_1)T$. So $|G|=|(R\cap P_1)T|\le |(R\cap P_1)||T|<|R||T|=|G|$, a contradiction.
Assume that $R$ is of order $p$. Since $R$ is the only minimal normal subgroup of $G$ and
$R\nsubseteq \Phi(G)$, $R=C_G(R)$. Thus $G/R=G/C_G(R)$ is isomorphic to some subgroup
of Aut$(R)$ which is a group of order $p-1$. So $G/R$ is a $p'$-group and $R=P$.
But this implies that $G=N_G(P)$ is $p$-nilpotent and the lemma holds.

\end{proof}

\begin{lemma}\label{2-s}
    Let  $L$ be a normal subgroup of $G$ with $G/L$ is a $2$-group. Suppose $P$ is a Sylow
    $2$-subgroup of $G$. If every  maximal subgroup $P_1$ of $P$ is either has a $2$-nilpotent supplement in $G$
    or satisfies $P_1\cap L=1$, then $G$ is $2$-nilpotent.
\end{lemma}
\begin{proof}
   If $P\cap L$ is of order 2 then $L$ is of order $2n$, where $n$ is an odd number, and hence is 2-nilpotent.
   Let $R$ be the normal Hall $2'$-subgroup of $L$. Then $G/R$ is a 2-group and hence $G$ is 2-nilpotent.
   Assume $|P\cap L|>2$. Let $P_1$ be any maximal subgroup of $P$. Then $P_1\cap L$ is maximal in $P\cap L$ or
   $P_1\cap L=P\cap L$. Thus $P_1\cap L\ne 1$. It follows that every maximal subgroup of $P$ has a 2-nilpotent
   supplement in $G$. Let $P_1$ be a maximal subgroup of $P$ and $T_1$ be a 2-nilpotent supplement of $P_1$ in $G$
   with maximal order. Let $Q$ be the Hall $2'$-subgroup of $T_1$. Then $Q$ is also a Hall $2'$-subgroup of $G$ and
   $Q\unlhd T_1$. Thus $N_G(Q)$ is 2-nilpotent and $T_1\subseteq N_G(Q)$. The maximality of $T_1$ shows that
   $T_1=N_G(Q)$. If $P\cap T_1=P$ then $G=T_1$ is 2-nilpotent. Assume that $P\cap T_1<P$ and let $P_2$ be
   a maximal subgroup of $P$ with $P\cap T_1\le P_2$. Then $P_2$ has a 2-nilpotent supplement  $T_2$ in $G$.
   Clearly, $T_2$ contains a Hall $2'$-subgroup of $G$. By \cite[Theorem A]{Gross},  $T_2$ contains
   some conjugate of $Q$ and without loss of generality, we can assume $Q\le T_2$. It follows that $Q$
   is also the normal Hall $2'$-subgroup of $T_2$ and so $T_2\subseteq N_G(Q)=T_1$. Thus $G=P_2T_2=P_2T_1=P_2Q<PQ=G$,
   a contradiction and the lemma holds.

\end{proof}

\begin{lemma}\label{min}
    Let $P$ be  a normal $p$-subgroup of $G$, where $p$ is a prime dividing $|G|$. If
    every subgroups of $P$ of order $p$ or 4(when $P$ is a nonabelian 2-group) not having
    $p$-nilpotent supplement in $G$ is $\Pi$-normal in $G$, then $P\subseteq Z_\infty^{\frak U}(G)$.
\end{lemma}
\begin{proof}
Let $R$ be the $\frak{A}_{p-1}$-residual of $N$, where $\frak{A}_{p-1}$ is the class of all abelian
group of exponent dividing $p-1$. Then by Lemma \cite[Lemma 2.15]{pi}, $P\subseteq Z_\infty^{\frak U}(G)$
if and only if $P\subseteq Z_\infty(R)$. Hence
    if the lemma is not true then the set $\Gamma=\{L\le P\mid L\unlhd G, L\nsubseteq Z_\infty(R)\}$ is
     non-empty.   Choose $L$ to be an element in the set of minimal order and let $L/K$ be a chief factor of $G$.
     Then $K\subseteq Z_\infty(R)$ and $L/K$ is not central in $R$. Since $G/C_G(L/K)$
     is isomorphic to some subgroup of Aut$(L/K)$, if $L/K$ is cyclic then $G/C_G(L/K)$ is abelian of exponent
     dividing $p-1$ and it follows that $C_G(L/K)\subseteq R$.  This implies that $L\subseteq Z_\infty(R)$, a
     contradiction. So $L/K$ is noncyclic.
Let $C=C_R(K)$. Then $O^p(R)\subseteq C$ by \cite[A,(12.3)]{D-H-BOOK}. Clearly, $C$ is normal in $G$. Since $L/K$
 is a $G$-chief factor and $L\cap KC=(L\cap C)K\unlhd G$. $(L\cap C)K=K$ or $L$. If $(L\cap C)K=K$ then
 $L\cap C\subseteq K$ and hence $[L, C]\le K$. Thus $O^p(G)\le C\le C_R(L/K)$ and so $R/C_R(L/K)$ is a $p$-number.
 Then by \cite[Lemma 1.7.11]{GUO-BOOK}, we see that every $R$-chief factor between $L$ and $K$ is $R$-central,
  so $L\subseteq Z_\infty(R)$ since $K\subseteq Z_\infty(R)$,
  a contradiction. Assume that $(L\cap C)K=L$. Then $(L\cap C)/(K\cap C)=(L\cap C)/(L\cap C\cap K)\cong (L\cap C)K/K=L/K$
  is a noncyclic $G$-chief factor. Thus $L\cap C\in \Gamma$. The minimality of $L$ shows that  $L\subseteq C$ and
  hence $K\subseteq Z(L)$.

Let $a$, $b$ be elements of order $p$ in $L$. Suppose $p > 2$ or $P$
is abelian. Then $(ab)^p = a^pb^p[b, a]^{\frac{p(p-1)}{2}} = 1$.
Hence the product of  elements of order $p$ is  of order $p$ or 1
and so $\Omega=\{a\in L\mid a^p=1\}$ is a subgroup of $L$. If
$\Omega\subseteq K$, then all elements of $C$ with $p'$-order act
trivially on every element of $L$ of order $p$ since they act
trivially on $K$. It follows from \cite[IV, Satz
5.12]{Huppert-BOOK-I} that all elements in $C$ of $p'$-order act
trivially on $L$. Thus $O^p(C)=O^p(R)\subseteq C_R(L/K)$ and, as
above argument, $L\subseteq Z_\infty(R)$, a contradiction. If
$\Omega \nsubseteq K$, then $L=\Omega K$.

Choose an element $a$ in $\Omega\setminus K$ such that $\langle
a\rangle K/K\subseteq L/K\cap Z(G_p/K)$. Let $H=\langle a\rangle$.
If $H$ has a $p$-nilpotent supplement $U$ in $G$, then $HK/K$ has a
$p$-nilpotent supplement $UK/K$ in $G/K$. Thus
$G/K=(HK/K)(UK/K)=(L/K)(UK/K)$. Since $L/K$ is minimal normal in
$G/K$ and is abelian, $L/K\cap UK/K=1$ or  $L/K\subseteq UK/K$ and
$UK/K=G/K$. If $L/K\cap UK/K=1$, then $|L/K|=|G/K:UK/K|=
|HUK/K:UK/K|\le |H|=p$. It follows that $L/K$ is cyclic of order
$p$, which contradicts to the choice of $L/K$. If $L/K\subseteq
UK/K=G/K$, then $L/K$ is  cyclic since $L/K$ is minimal normal in
$G/K$ and $G/K=UK/K\cong U/U\cap K$ is $p$-nilpotent. Hence $H$ has
no $p$-nilpotent supplement in $G$. Since $a$ is of order $p$, by
the hypotheses, $H$ is $\Pi$-normal in $G$ and so there is a
subnormal subgroup $T$ of $G$ such that $G=HT$ and $H\cap T\le I\le
H$, where $I$ has $\Pi$-property in $G$. Since $H$ is of prime
order, $H\cap T=H$ or 1. If $H\cap T=H$, then $H=I$ has
$\Pi$-property in $G$. By Proposition \ref{1} (1), $HK/K$ has
$\Pi$-property in $G/K$ and hence is $\Pi$-normal in $G/K$. It
follows from Lemma \ref{2} that $L/K=HK/K\cap L/K=HK/K$ is cyclic, a
contradiction. Assume that $H\cap T=1$. Then $|G:T|=p$ and, since
$T$ is subnormal in $G$, we see  that  $T$ is normal in $G$. Clearly
$G=LT$ and therefore $L/L\cap T\cong G/T$ is cyclic of order $p$.
Thus $G/C_G(L/L\cap T)$ is a group of exponent dividing $p-1$. It
follows that $L/L\cap T $ is a central factor of $R$. Since $|L\cap
T|<|L|$, $L\cap T\subseteq Z_\infty(R)$ by the choice of $L$. This
induce that $L\subseteq Z_\infty(R)$, a contradiction. This
contradiction shows that $\Gamma$ is empty and the lemma holds.

\end{proof}

\begin{lemma}\label{p-s1}
  Let $p$ be an odd prime and $P$ a Sylow $p$-subgroup of $G$. Assume that $N_G(P)$ is $p$-nilpotent. If
  any minimal subgroup of $P\cap O^p(G)$ either is contained in $Z_\infty(G)$ or has a $p$-nilpotent supplements in $G$.
   Then $G$ is $p$-nilpotent.
\end{lemma}
\begin{proof}
  Let $M$ be a proper subgroup of $G$ containing $P$. Then $O^p(M)=\langle x\in M\mid |x|$ is not divided by
  $p\rangle\subseteq \langle x\in G\mid |x|$ is not divided by $p\rangle=O^p(G)$. So any minimal subgroup of
  $P\cap O^p(M)$ either is contained in $Z_\infty(G)\cap M\subseteq Z_\infty(M)$ or has a $p$-nilpotent supplements
  in $G$ and so is in $M$. Since $N_M(P)=M\cap N_G(P)$ is $p$-nilpotent, the hypotheses hold on $M$ and $M$ is
  $p$-nilpotent by induction on $|G|$.
  Hence, as argument in Lemma \ref{p-s}, we have that $G$ is $p$-soluble and $O_{p'}(G)=1$.

By \cite[Proposition 2.3]{pi}, if a subgroup is contained in the hypercentral of $G$ then it has $\Pi$-property
in $G$. So $O_p(O^p(G))\subseteq Z_\infty ^{\frak U}(G)$ by Lemma \ref{min}. Since $O_{p'}(O^p(G))\le O_{p'}(G)=1$,
 $O_p(O^p(G))=F((O^p(G)))=F_p((O^p(G)))=F^*((O^p(G))$. By Lemma \ref{Fit}, $O^p(G)\subseteq Z_\infty ^{\frak U}(G)$
 and hence $G$ is $p$-supersoluble. Again by $O_{p'}(G)=1$, we have that $G$ is $p$-closed and $G=N_G(P)$ is
 $p$-nilpotent. The lemma holds.

\end{proof}

\section{Proofs of Theorems A and B}

 \noindent{\bf Proof of Theorem A} Assume that the theorem is not true and let $G$ be a counter example of minimal
 order. We divide the
proof into several steps.

(1) {it $O_{p'}(G)=1$.

If $O_{p'}(G)\ne 1$, then $PO_{p'}(G)/O_{p'}(G)$ is a Sylow
$p$-subgroup of $G/O_{p'}(G)$. For any subgroup $M/O_{p'}(G)$ of
$PO_{p'}(G)/O_{p'}(G)$ of order $p^m$, let $H=P\cap M$. Then
$M=HO_{p'}(G)$ and $H$ is of order $p^m$. If $H$ has a $p$-nilpotent
supplement $T$ in $G$, then $M/O_{p'}(G)$ has a $p$-nilpotent
supplement $TO_{p'}(G)/O_{p'}(G)$ in $G/O_{p'}(G)$, so if
$M/O_{p'}(G)$ has no $p$-nilpotent supplement in $G/O_{p'}(G)$ then
$H$ has no $p$-nilpotent supplement in $G$ and hence is $\Pi$-normal
in $G$ by the hypotheses. Therefore,
 by Lemma \ref{1}, we see that $M/O_{p'}(G)$ is $\Pi$-normal in $G/O_{p'}(G)$. Clearly,
 $N_{G/O_{p'}(G)}(PO_{p'}(G)/O_{p'}(G))=N_{G}(P)O_{p'}(G)/O_{p'}(G)$ is still $p$-nilpotent.
 It follows that the hypotheses  hold on
$G/O_{p'}(G)$. Hence $G/O_{p'}(G)$ is
$p$-nilpotent by the choice of $G$. It follows that $G$ is
$p$-nilpotent, a contradiction. So $O_{p'}(G)=1$.

(2) {\it Let $L$ be a minimal normal subgroup of $G$. Then $L$ is an abelian $p$-group. }

If $L$ does not contained in $N$ then $L\cong LN/N$ is a minimal normal subgroup of a $p$-nilpotent
group $G/N$. Thus $L$ is either an abelian $p$-group or a $p'$-group. But if $L$ is a $p'$-group then
$L\subseteq O_{p'}(G)=1$, a contradiction. Hence $L$ is abelian. Assume that $L\subseteq N$.  Since
$O_{p'}(G)=1$, $L$ is a $pd$-subgroup. Let $H$ be any subgroup of $P$ of order $p^m$ with $H\cap L\ne 1$.
 If $H$ has no $p$-nilpotent supplement in $G$, then $H$ is $\Pi$-normal in $G$. It follows from
 Lemma \ref{12} that $L$ is an abelian $p$-group. Assume that  any such subgroup $H$ has a $p$-nilpotent supplement
 in $G$ then so does in $LP$. Let $P_1$ be any maximal subgroup of $P$ with $P_1\cap L\ne 1$. Then $P_1$
 must contain some subgroup $H$ of order
$p^m$ and $H\cap L\ne 1$ and so $P_1$ has a supplement $p$-nilpotent in $LP$. Applying Lemma \ref{p-s},
we see that $LP$ is $p$-nilpotent and so is $L$. Since $O_{p'}(L)\subseteq O_{p'}(G)=1$, $L$ is an
abelian $p$-group and  (2) holds.

In the following, $L$ is always a minimal normal subgroup of $G$ contained in $N$.

(3) $|L|=p^m$

Let $G_p$ be a Sylow $p$-subgroup of $G$ containing $P$. If $|L|>
p^m$, then $L$ has a proper subgroup $H$ of order $p^m$ with
$H\unlhd G_p$. If $H$ is $\Pi$-normal in $G$, then by Lemma \ref{2},
$H=L$ or 1. This is nonsense by the choice of $H$. If $H$ has a
$p$-nilpotent supplement $T$ in $G$. Then $G=HT=LT$ and since $L$ is
minimal normal in $G$, $L\subseteq T$ or $L\cap T=1$. If $L\cap T=1$
then $|G|=|HT|\le |H||T|<|L||T|=|LT|=|G|$, a contradiction. Thus
$L\subseteq T$ and so $G=LT=T$ is $p$-nilpotent. This is contrary to
the choice of $G$. Therefore, $|L|\le p^m$.

Assume that
 $|L|\ne p^m$ then $|L|< p^m$. Let $|L|=p^l$. By a same argument as in (1), we can obtain that every subgroup
 of order $ p^{m-l}$ of $P/L$ having no $p$-nilpotent supplement in $G$ is $\Pi$-normal in $G$ and, clearly,
 $N_{G/L}(P/L)$ is $p$-nilpotent. Thus $G/L$ is $p$-nilpotent by the choice of $G$.
Since the class of all $p$-nilpotent subgroup is a saturated formation, we see that $L$ is the only minimal
 normal subgroup of $G$ contained in $N$ and $\Phi(G)\cap N=1$. Thus there is a maximal subgroup $M$ of $G$
  such that $G=L\rtimes M$ and $M\cong G/L$ is $p$-nilpotent. Since $\Phi(N)\subseteq \Phi(G)\cap N$, $\Phi(N)=1$
   and so $F(N)$ is abelian. It follows that $F(N)\cap M\unlhd \langle F(N), M\rangle= G$. But $L$ is the only
   minimal normal subgroup of $G$ contained in $N$ and $L\cap M=1$, so $F(N)\cap M=1$.
   Hence $F(N)=F(N)\cap LM=L(F(N)\cap M)=L$ and therefore, $L=O_p(N)=F(N)=C_N(L)$.

Since $G_p=G_p\cap LM=L\rtimes(G_p\cap M)$ is a Sylow $p$-subgroup
of $G$ containing $P$, $G_p\cap M$ is a Sylow $p$-subgroup of $M$.
Let $L_1$ be a maximal subgroup of $L$ with $L_1\unlhd G_p$ and
$H_1$ be a subgroup of $G_p\cap M\cap P$ of order $p^{m+1}/|L|$.
Then $H=L_1H_1$ is a subgroup of $P$ of
 order $p^m$.  If $H$ is $\Pi$-normal in $G$, then $L_1=H\cap L=L$ or 1 by Lemma \ref{2}. But $L_1$
  is proper in $L$, so $L_1=1$ and $L$ is of order $p$. It follows that $N/L=N/C_N(L)$ is isomorphic
  to some subgroup of Aut($L$) which is a group of order $p-1$, so $N/L$ is a $p'$-group. Hence $L=P$
  is the Sylow $p$-subgroup of $N$ and $G=N_G(L)$ is $p$-nilpotent, contrary to the choice of $G$.
  Thus by hypotheses, $H$ has a $p$-nilpotent supplement $U$ in $G$. Let $Q$ be a Hall $p'$-subgroup of $U$.
  Then $Q$ is also a $p'$-subgroup of $G$. Since $G=LM$ is $p$-soluble, we can assume that $Q$ is also a Hall
   $p'$-subgroup of $M$. Since $M$ is $p$-nilpotent, $M\subseteq N_G(Q)$. If $N_G(Q)=G$ then $Q\subseteq O_{p'}(G)=1$
   by (1) and so $G$ is a $p$-group, a contradiction. Thus, $M=N_G(Q)$ by the maximality of $M$. The $p$-nilpotency
    of $U$ implies that $U\le N_G(Q)$, so $U\le M$. Hence $G=HU=HM=L_1H_1M=L_1M$ and $|G|\le |L_1||M|<|L||M|=|LM|=|G|$,
     a contradiction. This contradiction shows that (3) is true.

(4) $m=1$

Since, by (3), $|L|=p^m$, we need only to prove that $L$ is of order
$p$. Assume that $L$ is noncyclic. We claim that all minimal
subgroup of $N/L$ of order $p$ having no $p$-nilpotent supplement in
$G$ is $\Pi$-normal in $G$. Assume $A/L$ is of order $p$ and
$A/L\subseteq N/L$. Clearly, $A$ is noncyclic since $L$ is
noncyclic. Thus there is a maximal subgroup $H$ of $A$ different
from $L$. Therefore, $A/L=HL/L$ and $|H|=|L|=p^m$. If $A/L$ has no
$p$-nilpotent supplement in $G/L$, then $H$ has no $p$-nilpotent
supplement in $G$. Hence, by hypotheses, $H$ is $\Pi$-normal in $G$.
Choose $T$ to be a subnormal subgroup of $G$ with $G=HT$ and $H\cap
T\le I$, where $I$ is a subgroup of $H$ having $\Pi$-property in
$G$. Since $|G:T|$ is a $p$-number, $O^p(G)\subseteq T$. If
$L\nsubseteq T$, then $L\nsubseteq O^p(G)$. It follows that $L\cong
LO^p(G)/O^p(G)$ is a cyclic chief factor of $G$. This is nonsense
because $L$ is noncyclic. Thus $L\subseteq T$. It follows that
$G/L=(A/L)(T/L)$ and $A/L\cap T/L=HL/L\cap T/L=(H\cap T)L/L\le
IL/L$. Since, By Lemma \ref{1}, $IL/L$ has $\Pi$-property in $G/L$,
$A/L$ is $\Pi$-normal in $G$ and our claim holds. It is easy to see
that $N_{G/L}(P/L)=N_G(P)/L$ is $p$-nilpotent. If $P/L$ is cyclic of
order $p$, then $N_{N/L}(P/L)=C_{N/L}(P/L)$ since
$N_{N/L}(P/L)=N/L\cap N_{G/L}(P/L)$ is $p$-nilpotent. It follows
from Lemma \ref{B-p} that $G/L$ is $p$-nilpotent. If $P/L$ is
noncyclic, then the hypotheses hold on $G/L$. Thereby, $G/L$ is
$p$-nilpotent by the choice of $G$. Now,  by a similar argument as
in (3), one can prove that $G$ is $p$-nilpotent. This contradicts to
the choice of $G$ and hence $L$ is cyclic of order $p$.

(5) $O_p(N)\subseteq Z_\infty ^{\frak U}(G)$.

By (2). $O_p(N)\ne 1$. If $|O_p(N)|=p$, then (5) is clear true. If $|O_p(N)|\ne p$,
then (5) holds from (4) and Lemma \ref{min}.

(6) {\it $N$ is not $p$-soluble.}

Assume that $N$ is $p$-soluble. Then since $O_{p'}(N)\subseteq O_{p'}(G)=1$,
$F^*(N)=F(N)=F_p(N)=O_p(N)\subseteq Z_\infty ^{\frak U}(G)$. By Lemma \ref{Fit},
$N$ is supersoluble. Again by $O_{p'}(N)=1$, we have that $P$ is normal in $N$ and
hence  is normal in $G$ for $P$ is the Sylow $p$-subgroup of $N$. It follows that $G=N_G(P)$ is $p$-nilpotent.

(7) {\it The final contradiction}

Let $R$ be the $\frak{A}_{p-1}$-residual of $N$, where $\frak{A}_{p-1}$ is
the class of all abelian group of exponent dividing $p-1$. Then $R$ is not
 $p$-soluble since $N$ is not. For any $N$-chief factor
$U/V$ with $U\subseteq O_p(N)$, $|U/V|=p$ by (5). Hence $N/C_N(U/V)$ is an
abelian group of order dividing $p-1$ and so $C_N(U/V)\subseteq R$. It
follows that $O_p(N)\subseteq Z_\infty (R)$. Let $O=O_p(N)\cap O^p(N)$. Then
$O\subseteq Z_\infty (R)$. Since $R$ is not $p$-soluble, $O<R\cap O^p(N)$.
Let $Q/O$ be a minimal normal subgroup of $G/O$ with $Q\subseteq R\cap O^p(N)$.
 Then $Q$ is not $p$-soluble. Otherwise, the minimality of $Q/O$ implies $Q/O$
 is a $p'$-group or a $p$-group. It follows that $Q$ is $p$-nilpotent since
  $O\subseteq Z_\infty(R)\cap Q\subseteq Z_\infty (Q)$ and so $Q\subseteq F_p(N)=F(N)=O_p(N)$,
   a contradiction. Let $H$ be any subgroup of order $p$ contained in $Q$. If $H$ has no $p$-nilpotent
   supplement in $G$ then $H$ is $\Pi$-normal in $G$ by the hypotheses and (4). Let $T$ be any subnormal supplement
    of $H$ in $G$. Then $O\subseteq O^p(G)\subseteq T$. So $G/O=(HO/O)(T/O)$ and $(HO/O)\cap (T/O)=(H\cap T)O/O$.
     Now, assume that $H\cap T\le I\le H$ with $I$ has $\Pi$-property in $G$. Then $(H\cap T)O/O\le IO/O\le HO/O$
     and, by Lemma \ref{1},
$IO/O$ has $\Pi$-property in $G$. Therefore, $HO/O$ is $\Pi$-normal in $G$.
If $H$ is not contained in $O$, then $(HO/O)\cap (Q/O)\ne 1$ and so $Q/O$ is abelian by Lemma \ref{12},
 a contradiction. Thereby, we have that for any subgroup $H$ of $Q$ of order $p$, if $H$ is not contained in $O$,
  $H$ has a $p$-nilpotent supplement in $G$. Let $X=QP$.
  Then $X\subseteq R$ and it follows from  $O\subseteq Z_\infty (R)$
  that $O\subseteq Z_\infty (X)$. Since $X/Q$ is a $p$-group, $O^p(X)=O^p(Q)\le Q$.
  Thus, by above, every minimal subgroup of order $p$ of $P\cap O^p(X)$ either
  is contained in $Z_\infty(X)$ or has a $p$-nilpotent supplement in $X$ and
  therefore, $X=QP$ is $p$-nilpotent by Lemma \ref{p-s1}, so $Q$ is $p$-soluble,
  a contradiction. This is the finial contradiction and the theorem holds.

 \noindent{\bf Proof of Theorem B}
We shall prove the theorem by induction on the order of $G$. By a similar argument as in proof of Theorem A,
$G/O_{p'}(G)$ satisfies the hypotheses of the theorem. Hence, if $O_{p'}(G)\ne 1$ then $G/O_{p'}(G)$ is
$p$-nilpotent by induction and so is $G$. Now assume that $O_{p'}(G)=1$. Then

(1) {\it Let $L$ be a minimal normal subgroup of $G$. Then $L$ is an abelian $p$-group.}

If $L\nsubseteq N$ then $L\cong LN/N\unlhd G/N$.  But $G/N$ is $p$-nilpotent, and hence $L$ is.
 Since $O_{p'}(G)=1$, $L$ is a $p$-group and so is abelian. Assume $L\subseteq N$.
 If $P$ possesses a subgroup $H$ of order $p^m$ with $H\cap L\ne 1$ and $H$ is $\Pi$-normal in $G$.
 Then by Lemma \ref{12} $L$ is a $p$-group and is abelian. Assume that every subgroup $H$ of $P$ of
  order $p^m$ with $H\cap L\ne 1$ has a $p$-nilpotent supplement in $G$. Let $P_1$ be any maximal
   subgroup of $P$ with $P_1\cap L\ne 1$. Then $P_1$ must contain a subgroup $H$ of order $p^m$ with $H\cap L\ne 1$.
   It follows that $P_1$ has a $p$-nilpotent supplement in $G$ and so does in $LP$.  Therefore,
    if $p=2$ then $LP$ is $p$-nilpotent by Lemma \ref{2-s} and so is $L$. It follows from  $O_{p'}(G)=1$
    that $L$ is a $p$-group and is abelian. Assume $p>2$. Then $G$ is of odd order since $(|G|,p-1)=1$.
    Thus $G$ is soluble and hence $L$ is abelian and is a $p$-group by $O_{p'}(G)=1$.

(2) {\it If $p=2$ and every cyclic subgroup of order 2 or 4(when $P$ is nonabelian) of $P$ having
no $p$-nilpotent supplement in $G$ is $\Pi$-normal, then $G$ is $p$-nilpotent.}

$O_p(N)>1$ by (1). If $O_p(N)$ is of order 2 then $O_p(N)\subseteq Z(G)\subseteq Z_\infty(G)$.
If $O_p(N)>2$ then $O_p(N)\subseteq Z_\infty(G)$ holds by
Lemma \ref{min}. Since $O_{p'}(G)=1$, $F(N)=O_p(N)\subseteq Z_\infty(G)$. If $F^*(N)=F(N)$ then
 $N\subseteq Z_\infty^{\frak U}(G)$. and hence $G$ is
$p$-supersoluble since $G/N$ is. But since $(|G|, p-1)=1$, that $G$ is $p$-supersoluble implies that
it is $p$-nilpotent. Assume that $F^*(N)\ne F(N)$
and let $R/F(N)$ be a $G$-chief factor with $R\subseteq F^*(N)$. Then $R$ is not soluble, otherwise $R$
is nilpotent since $F(N)=O_p(N)\subseteq
Z_\infty(G)$ and so $R\subseteq F(N)$, a contradiction. Let $O=O^p(R)$. Then $R=OF(N)$ and $O$ is not soluble.
 If there is an element $a\in O\setminus F(N)$ of order 2, then $H=\langle a\rangle $ does not
 avoid $R/F(N)$. If $H$ has a $p$-nilpotent supplement  $T$ in $G$, then since $p=2$, $T$ is soluble and
 since $|H|=2$, $T=G$ or
 $|G:T|=2$. But in both case we have $G$ is soluble and so $F^*(N)=F(N)$, a contradiction.
 Assume that $H$  is $\Pi$-normal in $G$. Let $T$ be a subnormal supplement of $H$ in $G$.
 Then $O\subseteq O^p(G)\subseteq T$ since $|G:T|=|H|=p$. Thus $H$ has $\Pi$-property in $G$ and so $HF(N)/F(N)$
 has $\Pi$-property in $G/F(N)$ by Lemma \ref{1}. Since $1\ne HF(N)/F(N)=HF(N)/F(N)\cap R/F(N)$, $R/F(N)$ is
 a $p$-group and is abelian by Lemma \ref{12}, a contradiction. Thus all elements of order 2 of $O$ are lying $F(N)$.
 If there is an element $a\in O\setminus F(N)$ of order 4, then $a^2\in F(N)$ since $|a^2|=2$. By a similar argument,
 one can find a contradiction.
Thus all elements of order 2 or 4 of $O$ are lying
$F(N)=O_p(N)\subseteq Z_\infty (G)$. Since $O$ is not soluble, $O$
has a minimal non-2-nilpotent subgroup $X$. Then $X=A\rtimes B$,
where $A$ is a $p$-group of exponent $p$ or 4 (when $A$ is
nonabelian) and $B$ is a $p'$-group. But all elements of order 2 or
4 of $O$ is in $F(N)$,  so $A\subseteq F(N)\subseteq Z_\infty (G)$.
It follows that $B$ acts trivially on $A$, a contradiction.  Thus
$F^*(N)=F(N)$ and (2)
 holds.

(3){\it Let $L$ be a normal subgroup of $G$ contained in $N$. If $|L|=p$ and $p=m=2$ then $G$ is $p$-nilpotent.}

Let $H$ be a cyclic subgroup of order 4. Then, since $p=m=2$, $H$ has a $p$-nilpotent supplement in $G$ or
is $\Pi$-normal in $G$. Now let $A=\langle a\rangle$ be a subgroup of $P$ of order 2. Then $LA$ is of order 4
and hence either has a $p$-nilpotent supplement in $G$ or is $\Pi$-normal in $G$.
Assume $LA$ has a $p$-nilpotent supplement $T$ in $G$. Since $|L|=p=2$, $L\subseteq Z(G)$. It follows that $LT$
is still $p$-nilpotent and we can assume that $T=LT$ and so $G=AT$. Thereby $|G:T|=|AT:T|\le 2$. This means that
$T$ is normal in $G$ and the normal Hall $p'$-subgroup of $T$ is also the normal Hall $p'$-subgroup of $G$.
Thus $G$ is $p$-nilpotent.

Assume that $LA$ is $\Pi$-normal in $G$. Then there is a subnormal subgroup $T$ of $G$ with $LA\cap T\le I\le LA$
and $I$ has $\Pi$-property in $G$.
We claim that $A$ is $\Pi$-normal in $G$. Assume that $I=LA$ and then $T=G$. Let $U/V$ be any $G$-chief factor.
If $LAV\le U$ then $LV/V\le U/V$. Since $LV\unlhd G$ and $U/V$ is a $G$-chief factor, $LV=U$ or $V$. If $LV=U$
then $U/V$ is cyclic. Hence $AV/V\cap U/V=U/V$ or 1. So $AV/V\cap U/V\unlhd G/V$ and $|G/V:N_{G/V}(AV/V\cap U/V)|=1$.
 If $LV=V$
then $L\subseteq V$ and $LAV\cap U=AV\cap U$. Thus $|G/V:N_{G/V}(AV/V\cap U/V)|=|G/V:N_{G/V}(LAV/V\cap U/V)|$
is a $2$-number since $LA=I$ has $\Pi$-property in $G$. Assume that $LAV\nsubseteq U$ then $ALV/V\cap U/V<ALV/V$
and so $ALV/V\cap U/V$ is of order 1 or 2 since $|ALV/V|\le |LA|=4$. It follows that $AV/V\cap U/V=ALV/V\cap U/V$ or 1.
 If $AV/V\cap U/V=1$ then $|G/V:N_{G/V}(AV/V\cap U/V)|=1$. If $AV/V\cap U/V=ALV/V\cap U/V$,
 then $|G/V:N_{G/V}(AV/V\cap U/V)|=|G/V:N_{G/V}(LAV/V\cap U/V)|$ is a 2-number since $LA=I$ has $\Pi$-property in $G$.
  Thus in this case we have that $A$ has $\Pi$-property in $G$ and so is $\Pi$-normal in $G$. Assume that $I<LA$.
  Then $|I|=1$ or 2. If $|I|=1$ then $LA\cap T=1$. It follows that $|G:T|=4$ and $LT<G$ and so $A\cap LT=1$.
  Since $G=LAT=A(LT)$ and $LT$ is normal in $G$, $A$ is $\Pi$-normal in $G$. Assume $|I|=2$. Then $|G:T|=2$.
   If $A\nsubseteq T$ then $G=AT$ and $A\cap T=1$. This implies that $A$ is $\Pi$-normal in $G$. If $A\subseteq T$
   then $A=LA\cap T=I$ has $\Pi$-property in $G$ and hence is $\Pi$-normal in $G$. Thus our claim holds.

By (2), we see that $G/L$ is $p$-nilpotent and so is $G$ since $|L|=p=2$. Hence (3) holds.

(4) {\it Let $L$ be a normal subgroup of $G$ contained in $N$. If $|L|\ne p^m$, then $G$ is $p$-nilpotent.}

If $|L|>p^m$ then $L$ has a proper subgroup $H$ of order $p^m$ with $H\unlhd G_p$, where $G_p$ is some Sylow
 $p$-subgroup containing $P$. Since $L$ is minimal normal in $G$ and $H<L$, $G$ is the only supplement of $H$
 in $G$. So, if $H$ has a $p$-nilpotent supplement in $G$, then $G$ is $p$-nilpotent. Assume that $H$ is $\Pi$-normal
 in $G$. Then by Lemma \ref{2} $H=H\cap L=L$ or 1. This is nonsense and so we can assume that $|L|\le p^m$.

Suppose $|L|<p^m$. We claim that the hypotheses still hold on $G/L$. Let $|L|=p^l$. For any subgroup $H/L$ of order
 $p^{m-l}$ of $P$, $H$ is a subgroup of order $p^m$ of $P$. Thus $H$ is $\Pi$-normal in $G$ if it has no $p$-nilpotent
  supplement in $G$. Clearly if $H/L$ has no $p$-nilpotent supplement in $G/L$ then $H$ has no $p$-nilpotent supplement
   in $G$, so, by Lemma \ref{1} (3), if $H/L$ has no $p$-nilpotent supplement in $G/L$ then it is $\Pi$-normal in $G/L$,
    that is, every subgroup of order $p^{m-l}$ of $P/L$ having no $p$-nilpotent supplement is $\Pi$-normal in $G$.
If $p>2$ or $P$ is abelian or $m-l\ne 1$ then condition (i) or (ii) holds on $G/L$ and our claim holds in this case.

Now assume that $p=2$, $m-l=1$. If $l=1$ then $G$ is $p$-nilpotent by by (3). Assume that $l>1$. Let $H/L$ be any
cyclic subgroup of $P/L$ of order 4. Then there is an element $a\in H$ with $H=L\langle a\rangle$. Since $a^4\in L$,
$|a|=4$ or 8.  Let $L_1$ be a maximal subgroup of $L$. Since $m=l+1$, if $|a|=4$, then  $\langle a\rangle L_1$ is
of order $p^m$. Assume $|a|=8$. Since $l>1$, $L_1\ne 1$ and we can choose a maximal subgroup $L_2$ of $L_1$.
Then either $\langle a\rangle L_1$  or $\langle a\rangle L_2$ is of order $p^m$. Thus in any case, $H$ has a
subgroup $H_1$ of order $p^m$ containing $\langle a\rangle$. By the hypotheses, if $H_1$ has no $p$-nilpotent
 supplement in $G$ then is $\Pi$-normal in $G$. If $H_1$ has a $p$-nilpotent supplement in $G$ then so does $H$
 and therefore, $H/L$ has a $p$-nilpotent supplement in $G/L$. Assume that $H_1$ is $\Pi$-normal in $G$. Then
  there is a subnormal subgroup $T$ of $G$ with $H_1T=G$ and $H_1\cap T\le I\le H_1$, where $I$ has $\Pi$-property
   in $G$. Since $|G:T|$ is a $p$-number, $O^p(G)\subseteq T$. If $L\nsubseteq T$ then $L\nsubseteq O^p(G)$
   and $L\cong LO^p(G)/O^p(G)$ is a chief factor of $G$ and hence is cyclic. This is contrary to $l>1$.
   So $L\subseteq T$ and hence $H/L\cap T/L=(H_1L/L)\cap T/L=(H_1\cap T)L/L\le IL/L\le H/L$.
   By Lemma \ref{1}, $IL/L$ has $\Pi$-property in $G$, so $H/L$ is $\Pi$-normal in $G$.
   Thus condition (iii) holds on $G/L$. So, in any case, hypotheses still hold on $G/L$ and our claim is true.

By Induction, we have $G/L$ is $p$-nilpotent. If $L\in \Phi(G)$ then $G$ is $p$-nilpotent
by \cite[Lemma 1.8.1]{GUO-BOOK}. Assume $L\nsubseteq \Phi(G)$. Then $G$ has a maxima subgroup
 $M$ with $G=L\rtimes M$. It follows that $M\cong G/L$ is $p$-nilpotent. Let $Q$ be the Hall
  $p'$-subgroup. Then $Q\unlhd M$. Since $M$ is maximal in $G$ and $O_{p'}(G)=1$, $M=N_G(Q)$.
   If $L$ is cyclic, then $L\subseteq Z(G)$ since $(|G|,p-1)=1$ and then $G$ is $p$-nilpotent
   by $G/L$ is.  Assume that $L$ is noncyclic. Then $L$ has a maximal subgroup $L_1$ with $L_1\ne 1$
   and $L_1\unlhd G_p$, where $G_p$ is a Sylow $p$-subgroup of $G$. Let $H_1$ be a subgroup of $P\cap M$
   of order $p^{m-l+1}$. Then $H=L_1H_1$ is of order $p^m$. If $H$ is $\Pi$-normal in $G$.
   Then $L_1=L\cap H=L$ or 1 by Lemma \ref{2}. This is nonsense and so $H$ has a $p$-nilpotent supplement in $G$.
   Let $T$ be a $p$-nilpotent supplement of $H$ in $G$ and $Q_1$ is the Hall $p'$-subgroup of $T$. If $p>2$ then $G$
    is of odd order by $(|G|,p-1)=1$ and is soluble. Thus $Q$ is conjugate with $Q_1$. If $p=2$ then
    by \cite[Theorem A]{Gross}, $Q$ is conjugate with $Q_1$. Since any conjugate of $T$ is also a
    $p$-nilpotent supplement of $H$ in $G$, we can assume $Q\subseteq T$. It follows that $T\subseteq N_G(Q)=M$
    and so $G=HT=HM=L_1M<LM=G$. This is nonsense and so (4) holds.

(5) {\it  Let $L$ be a minimal normal subgroup of $G$ contained in
$N$. Then $|L|=p$.}

Assume that $|L|>p$ and let $X/L$ be a minimal subgroup of $P/L$.
Then $X=L\langle x\rangle$ for some $x\in X$. Let $\mho=\langle
a^p\mid a\in X\rangle\cap L$. Then $\mho\le \Phi(X)\le L$ since $L$
is maximal in $X$. If $\mho=L$ then $\mho=\Phi(X)=L$ and hence $X$
is cyclic because $X/L=X/\Phi(X)$ is. This imply that $L$ is cyclic,
a contradiction. Thus $\mho<L$. Let $L_1$ be a maximal subgroup of
$L$ with $\mho\subseteq L_1$. Then $x^p\in \mho\subseteq L_1$ and
$x\notin L_1$. It follows that $H=L_1\langle x \rangle$ is of order
$|L|=p^m$. If $H$ is $\Pi$-normal in $G$, then there is a subnormal
subgroup $T$ of $G$ with $HT=G$ and $H\cap T\le I\le H$, where $I$
has $\Pi$-property in $G$. Since $|G:T|$ is a $p$-number,
$O^p(G)\subseteq T$. If $L\nsubseteq T$ then $L\nsubseteq O^p(G)$
and $L\cong LO^p(G)/O^p(G)$ is a chief factor of $G$ and hence is
cyclic. This is contrary to $|L|>p$. So $L\subseteq T$ and hence
$X/L\cap T/L=(HL/L)\cap T/L=(H\cap T)L/L\le IL/L\le X/L$. By Lemma
\ref{1}, $IL/L$ has $\Pi$-property in $G$, so $X/L=HL/L$ is
$\Pi$-normal in $G$.  If $H$ has a $p$-nilpotent
 supplement in $G$ then so does $X$ and hence $X/L$ has a $p$-nilpotent supplement in
 $G/L$. Thus if $p$ is odd then the hypotheses hold on $G/L$. If
 $p=2$, similarly, one can prove that every cyclic group of $P/L$ of order 4 either is $\Pi$-normal
   or has a $p$-nilpotent supplement in $G/L$. It follows from step (2) when $p=2$ that $G/L$ is $p$-nilpotent
    and  by induction on $|G|$ when $p$ is odd that $G/L$ is
    $p$-nilpotent.
Now, since $G/L$ is $p$-nilpotent and $L$ is noncyclic, by a similar
argument as in (4), we can obtain a contradiction and hence (5)
holds.

(6) {\it  If (i) holds, then $G$ is $p$-nilpotent.}

It directly from (4) and (5).

(7) {\it If (ii) holds, then $G$ is $p$-nilpotent}.

By (6), we can assume that $m=1$. If $p=2$ and $P$ is abelian, then it is directly from (2)
 that $G$ is $p$-nilpotent. If $p>2$ then $G$ is of odd order and is soluble.
 By Lemma \ref{min}, $F(N)=O_p(N)\subseteq Z^{\frak U}_\infty(G)$. It follows from Lemma \ref{Fit}
  and $(|G|,p-1)=1$ that $G$ is $p$-nilpotent.

(8) {\it If (iii) holds, then $G$ is $p$-nilpotent}.

By (6), $m=1$ and the $p$-nilpotency of $G$ is directly from (2).

(9) {\it If (iv) holds, then $G$ is $p$-nilpotent}.

by (6) and (7), we can assume that $p=2$ and $m=1$.   Let $M/N$ be the 2-complement of $G/N$.
Then $M$ is soluble since $N$ is soluble and $M/N$ is of odd order. Also, $P$ is a Sylow $p$-subgroup
of $M$. By \cite[Theorem 2.8]{Dor-quterian-free},  $P\cap Z(M)\cap K_\infty (M)=1$, where $K_\infty (M)$
 is the final term of the lower central series of $N$. If $ K_\infty (M)\cap N\ne 1$ and let $L$ be a minimal
 normal subgroup of $M$ contained in $ K_\infty (M)\cap N$ then, by (5), $L$ is of order 2 and hence is contained in
  $Z(M)$, a contradiction. so $ K_\infty (M)\cap N=1$. This implies that $ K_\infty (M)$ is a $p'$-group and so
  $K_\infty (M)\subseteq O_{p'}(G)=1$. That is, $M$ is nilpotent and hence is a $p$-group by $O_{p'}(G)=1$.
  It follows that $G$ is a $p$-group and is $p$-nilpotent.
 Hence (9) holds and the proof is completed.

\section{Some Remarks, Examples and Corollaries}

1. In Theorem A, if the hypothesis ``$N_G(P)$ is $p$-nilpotent'' is deleted,
we can prove similarly that $G$ is $p$-supersoluble when $G$
is $p$-soluble. But in this case,  ``$G$ is $p$-soluble'' must be requested. Otherwise, we have the following example.

\begin{example}
    Let $G=A_5\times Z_5$, where $A_5$ is the alternative group of degree 5
    and $Z_5$ is a cyclic group of order 5. Then $Z_5$ is the only subgroup
    of $G$ of order 5 which has no 5-nilpotent supplement in $G$.
    Clearly $Z_5$ is normal in $G$ but $G$ is not 5-soluble.
\end{example}

2. By choose $m$ to be some special number, one can obtain some corollaries of Theorem A.
 For example, when $\frac{|P|}{p^m}=p$, or in the same, subgroups of order
  $p^m$ of $P$ are maximal in $P$, we have the following corollary.

\begin{corollary}\label{c1}
     Let $p$ be an odd prime and $N$ a normal subgroup of $G$
      with $p$-nilpotent quotient. Assume that $P$ is a Sylow
       $p$-subgroup of $N$ and $N_G(P)$ is $p$-nilpotent. If every maximal subgroup  of $P$
       not having $p$-nilpotent supplement in $G$ is $\Pi$-normal in $G$, then $G$ is $p$-nilpotent.
\end{corollary}
\begin{proof}
    Let $P_1$ be a maximal subgroup of $P$ and assume $|P_1|=p^m$.
    If $P$ is not of order $p$, then $1<p^m<|P|$. Since all maximal
     subgroups of $P$ are of same order, the hypotheses of Theorem A hold.
     Thus $G$ is $p$-nilpotent. If $|P|=p$, then $P\subseteq C_G(P)$.
     It follows from $N_G(P)$ is $p$-nilpotent that  $N_G(P)=C_G(P)$.
     By Lemma \ref{B-p}, we see that $G$ is $p$-nilpotent.
\end{proof}

Assume subgroup of order $p^m$ is 2-maximal in $P$, that is,
it is a maximal subgroup of a maximal subgroup of $P$. Then we have
\begin{corollary}
     Let $p$ be an odd prime and $N$ a normal subgroup of $G$
     with $p$-nilpotent quotient. Assume that $P$ is a Sylow
     $p$-subgroup of $N$ and $N_G(P)$ is $p$-nilpotent.
     If every 2-maximal subgroup  of $P$  not having $p$-nilpotent supplement
     in $G$ is $\Pi$-normal in $G$, then $G$ is $p$-nilpotent.
\end{corollary}
\begin{proof}
    Let $P_1$ be a 2-maximal subgroup of $P$ and assume $|P_1|=p^m$.
    Then $|P:P_1|=p^2$. If the order of $P$ is greater than $p^2$,
    then the hypotheses of Theorem A hold. Thus $G$ is $p$-nilpotent.
    If $|P|=p^2$ or $p$, then since any group of order $p^2$ or $p$ is abelian,
    it can be proved that $G$ is $p$-nilpotent by argument as in Corollary \ref{c1}.
\end{proof}

Similarly, if $m=1$ or 2,  we can obtain

\begin{corollary}
    Let $p$ be an odd prime and $N$ a normal subgroup of $G$ with
    $p$-nilpotent quotient. Assume that $P$ is a Sylow $p$-subgroup
     of $N$ and $N_G(P)$ is $p$-nilpotent. If every minimal subgroup
      of $P$  not having $p$-nilpotent supplement in $G$ is $\Pi$-normal in $G$, then $G$ is $p$-nilpotent.
\end{corollary}

\begin{corollary}
    Let $p$ be an odd prime and $N$ a normal subgroup of $G$ with
     $p$-nilpotent quotient. Assume that $P$ is a Sylow $p$-subgroup
     of $N$ and $N_G(P)$ is $p$-nilpotent. If every 2-minimal subgroup
      of $P$  not having $p$-nilpotent supplement in $G$ is $\Pi$-normal in $G$, then $G$ is $p$-nilpotent.
\end{corollary}

3. Corollaries of Theorem B.

\begin{corollary}
    Let be $G$ be a group and $p$ a prime with $(|G|, p-1)=1$. Assume that $N$ is a
     normal subgroup of $G$ with $p$-nilpotent quotient and $P$  a Sylow $p$-subgroup of $N$.
      Suppose that  every maximal subgroup of $P$ not having $p$-nilpotent supplement in $G$
       is $\Pi$-normal in $G$, then $G$ is $p$-nilpotent.
\end{corollary}

\begin{corollary}
    Let be $G$ be a group and $p$ a prime with $(|G|, p-1)=1$. Assume that $N$ is a normal
    subgroup of $G$ with $p$-nilpotent quotient and $P$  a Sylow $p$-subgroup of $N$.
    If every subgroup of $P$ of order $p$ or 4(when $P$ is nonabelian 2-group) not having
    $p$-nilpotent supplement in $G$ is $\Pi$-normal in $G$ then $G$ is $p$-nilpotent.
\end{corollary}

\begin{corollary}
    Let  $N$ be a soluble normal subgroup of $G$ with $2p$-nilpotent quotient and $P$  a Sylow $2$-subgroup of $N$.
    If $P$ is quaternion-free and every subgroup of $P$ of order $2$
     not having $p$-nilpotent supplement in $G$ is $\Pi$-normal in $G$, then $G$ is $2$-nilpotent.
\end{corollary}

\begin{corollary}\label{cm1}
    Let be $G$ be a group and $p$ a prime with $(|G|, p^2-1)=1$. Assume that $N$ is a normal subgroup of
    $G$ with $p$-nilpotent quotient and $P$  a Sylow $p$-subgroup of $N$. If  every 2-maximal subgroup of $P$
    not having $p$-nilpotent supplement in $G$ is $\Pi$-normal in $G$, then $G$ is $p$-nilpotent
\end{corollary}

\begin{corollary}\label{cm2}
    Let be $G$ be a group and $p$ a prime with $(|G|, p^2-1)=1$. Assume that $N$ is a normal subgroup of $G$
     with $p$-nilpotent quotient and $P$  a Sylow $p$-subgroup of $N$. Suppose that  every 2-minimal subgroup
     of $P$ not having $p$-nilpotent supplement in $G$ is $\Pi$-normal in $G$, then $G$ is $p$-nilpotent
\end{corollary}

In Corollaries \ref{cm1} and \ref{cm2}, the hypothesis ``$(|G|, p^2-1)=1$'' can not be replaced by
``$(|G|, p-1)=1$'', for example $G=N=A_4$ and $p=2$. Moreover, counter example also exists when $p$ is odd.

\begin{example}
 Let $G = \langle a, b\rangle\rtimes \langle x\rangle$, where $a^5 = b^5 = x^3 = 1, ab = ba$ and
 $a^x = ab, b^x = a^2b^3$. Then $G$ is not 5-nilpotent. Let $P=\langle a, b\rangle$. Then $P$
 is the Sylow 5-subgroup of $G$. The 2-minimal subgroup is $P$ and 2-maximal subgroup of $P$ is 1,
  which are both normal in $G$.

\end{example}

The following lemma shows that if $p$ is minimal and is odd, then the  hypothesis ``$(|G|, p^2-1)=1$''
can be delete, and when $p=2$, this hypothesis can be replaced by $G$ is $A_4$-free.

\begin{lemma}{\rm (\cite[Lemma 2.8]{L-SK-SC})}
    Let $p$ be the minimal prime divisor of the order of a group $G$. Assume that
$G$ is $A_4$-free and $L$ is a normal subgroup of $G$. If $G/L$ is $p$-nilpotent and $p^3 \nmid |L|$, then $G$ is
$p$-nilpotent.
\end{lemma}

4. In the literature one can also find  many special cases of Theorems A and B, for example:

\begin{corollary}{\rm(\cite[Theorem 3.1]{Han-p-nil})}
    Let $p$ be an odd prime dividing the order of $G$ and  $P$  a Sylow $p$-subgroup of $G$.
     Suppose $N_G(P)$ is $p$-nilpotent and  there exists a subgroup $D$ of $P$ with $1<|D|<|P|$
     such that every subgroup $H$ of $P$ with order $|D|$  is s-semipermutable in $G$. Then $G$ is $p$-nilpotent.
\end{corollary}

\begin{corollary}{\rm(\cite[Theorem 3.2]{Han-p-nil})}
    Let $p$ be a priming dividing the order of $G$ satisfying
$(|G|, p-1) = 1$ and $P$ be a Sylow $p$-group of $G$. Suppose there exists a
nontrivial subgroup $D$ of $P$ such that $1 < |D| < |P|$ and every subgroup $H$
with order $|D|$ and $2|D|$ (if $P$ is a non-abelian 2-group and $|P : D| > 2$) is
s-semipermutable in $G$, then $G$ is a $p$-nilpotent group.
\end{corollary}

\begin{corollary}{\rm(\cite[Theorem 3.1]{WangLF})}
     Let $p$ be an odd prime dividing the order of $G$ and  $P$  a Sylow $p$-subgroup of $G$.
     If $N_G(P)$ is $p$-nilpotent and  every maximal subgroup of $P$   is s-semipermutable in $G$.
     Then $G$ is $p$-nilpotent.
\end{corollary}

\begin{corollary}{\rm(\cite[Theorem 3.3]{WangLF})}
     Let $p$ be the smallest prime number dividing the order of $G$ and  $P$  a Sylow $p$-subgroup of $G$.
      If  every maximal subgroup of $P$   is s-semipermutable in $G$. Then $G$ is $p$-nilpotent.
\end{corollary}

\begin{corollary}{\rm(\cite[Theorem 3.5]{WangLF})}
     Let $p$ be the smallest prime number dividing the order of $G$ and  $P$  a Sylow $p$-subgroup of $G$.
     If  every 2-maximal subgroup of $P$   is s-semipermutable in $G$ and $G$ is $A_4$-free. Then $G$ is $p$-nilpotent.
\end{corollary}

\end{document}